\documentclass[preprint,12pt]{elsarticle}
\usepackage{amsfonts}
\usepackage{amsmath}
\usepackage{amssymb,amsmath,amsthm,amsfonts,eepic,graphicx}

\usepackage{picinpar,picins,subfigure}
\usepackage{fancyhdr}
\usepackage{fullpage}
\usepackage{color}

\usepackage{tabularx}

\newtheorem{theorem}{Theorem}
\newtheorem{proposition}{Proposition}
\newtheorem{lemma}{Lemma}

\newtheorem{claim}{Claim}

\begin{document}
\begin{frontmatter}

\title{Asymptotic  Critical   Radii in Random Geometric Graphs over 3-Dimensional Convex regions}

\author{Jie Ding} \ead{jieding78@hotmail.com}
%\address{School of Computer Science, Jiangsu University of Science and Technology, Zhenjiang 212100, China
% \\ China Institute of FIZ Supply Chain, Shanghai Maritime University, Shanghai 201306, China}
\address{China Institute of FIZ Supply Chain, Shanghai Maritime University, Shanghai 201306, China}

\author{Xiaohua Xu} \ead{artex@gmail.com}
\address{School of Information Engineering,
Yangzhou University, Yangzhou 225000, China}

\author{Shuai Ma} \ead{mashuai@buaa.edu.cn}
\address{ SKLSDE Lab,
Beihang University, Beijing 100191, China}

\author{Xinshan Zhu} \ead{xszhu126@126.com}
\address{
School of Electrical and Information Engineering, Tianjin University, Tianjin 300072, China}

\begin{abstract}
This article presents  the precise asymptotical distribution of two types of critical transmission radii, defined in terms of
$k-$connectivity and the minimum vertex degree, for a random geometry graph distributed over a 3-Dimensional Convex region.
\end{abstract}


\end{frontmatter}

\section{Introduction and main results}
\label{sec-intro}

\par Let $\chi_n$ be a uniform $n$-point process over a
  convex region $\Omega\subset\mathbb{R}^d(d\geq2)$, i.e., a set of $n$ independent
points each of which is uniformly distributed over $\Omega$, and every pair of points whose Euclidean distance less than $r_n$ is
connected with an undirected edge. So  a random geometric
graph $G(\chi_n,r_n)$ is obtained.

 $k-$connectivity and
the smallest vertex degree are two interesting topological
properties of a random geometry graph. A graph $G$ is said to be
$k-$connected if there is no set of $k-1$ vertices whose removal
would disconnect the graph. Denote by $\kappa$ the connectivity of
$G$, being the maximum $k$ such that $G$ is $k-$connected.
The minimum vertex degree of $G$ is denoted by $\delta$.   Let
$\rho(\chi_n;\kappa\geq k)$ be the minimum $r_n$ such that
$G(\chi_n,r_n)$ is $k-$connected and $\rho(\chi_n;\delta\geq k)$ be
the minimum $r_n$ such that $G(\chi_n,r_n)$ has the smallest degree $k$, respectively.

When $\Omega$ is a unit-area convex region on $\mathbb{R}^2$, the precise probability distributions of these two types of critical radii have been given in an asymptotic manner:
\begin{theorem}\label{thm:combined-Main}(\cite{Ding-RGG2018,Ding-RGG2021})
Let $\Omega\subset \mathbb{R}^2$ be a unit-area convex region such that the length of
the boundary $\partial\Omega$  is $l$, $k\geq 0$ be an integer and $c>0$ be a constant.
\par (i) If $k>0$, let
     \begin{equation*}\label{eq:Theorem-formula-1}
        r_n=\sqrt{\frac{\log n+(2k-1)\log\log n+\xi}{\pi n }},
     \end{equation*}
where $\xi$ satisfies
$$ \left\{\begin{array}{cc}
     \xi=-2\log \left(\sqrt{e^{-c}+\frac{\pi l^2}{64}}-\frac{l\sqrt{\pi}}{8}\right), & k=1, \\
      \xi=2\log \left(\frac{l\sqrt{\pi}}{2^{k+1}k!}\right)+2c, & k>1. \\
    \end{array}
        \right.
$$
\par (ii) If $k=0$, let
     \begin{equation*}\label{eq:Theorem-formula-2}
        r_n=\sqrt{\frac{\log n+c}{\pi n }}.
     \end{equation*}
Then
    \begin{equation}\label{eq:Theorem-formula-3}
       \lim_{n\rightarrow\infty} \frac n{k!}\int_{\Omega}\left(n|B(x,r_n)\cap\Omega|\right)^k e^{-n|B(x,r_n)\cap\Omega|}dx= e^{-c},
    \end{equation}
and therefore,  the probabilities of the two events
$\rho(\chi_n;\delta\geq k+1)\leq r_n$ and $\rho(\chi_n;\kappa\geq
k+1)\leq r_n$ both converge to $\exp\left(-e^{-c}\right)$  as $n\rightarrow\infty$.
\end{theorem}
This theorem  firstly reveals how the region shape impacts on the critical transmission ranges,
  generalising the previous work~\cite{DH-largest-NN, Penrose-RGG-book,Penrose-k-connectivity,PWan04:Asymptotic-critical-transmission-MobileHoc, PJWan-IT-asymptotic-radius} in which only regular regions like disks or squares are considered.
This paper further demonstrates the asymptotic distribution of the critical radii for convex regions on $\mathbb{R}^3$:
\begin{theorem}\label{thm:Main3D}
Let $\Omega\subset \mathbb{R}^3$ be a unit-volume convex region such that the area of
the boundary $\partial\Omega$  is $\mathrm{Area}(\partial\Omega)$, $k\geq 1$ be an integer and $c>0$ be a constant.
Let
 \begin{equation}\label{eq:Theorem-radius}
 r_n=\left(\frac{16}{5\pi}\frac{\log n+(\frac{3k}{2}-1)\log\log n+\xi}{ n }\right)^{\frac13},
\end{equation}
where   $\xi$  solves  $$ \mathrm{Area}(\Omega)\frac{4}{3\pi}e^{-\frac{2\xi}{3}}\left(\frac{5\pi}{16}\right)^{\frac23}\left(\frac{2}{3}\right)^k\frac{1}{k!} =e^{-c},
$$
  then the probabilities of the two events
$\rho(\chi_n;\delta\geq k+1)\leq r_n$ and $\rho(\chi_n;\kappa\geq
k+1)\leq r_n$ both converge to $\exp\left(-e^{-c}\right)$  as $n\rightarrow\infty$.
\end{theorem}
The proof  of Theorem~\ref{thm:Main3D} follows the framework presented in~\cite{Ding-RGG2018,Ding-RGG2021}. However,
the details of the technique of boundary treatment are different.
To prove Theorem~\ref{thm:Main3D} it suffices to prove the following four propositions. In fact,
Theorem~\ref{thm:Main3D} is a consequence of  Proposition~\ref{pro:conclusion-1}
and~\ref{pro:conclusion-2}. However, the proofs of Proposition~\ref{pro:conclusion-1}
and~\ref{pro:conclusion-2} rely on
Proposition~\ref{pro:Explicit-Form} and Proposition~\ref{pro:Poisson-Version} which will be proved in Section~\ref{sec:Pro-1}
and~\ref{sec:Pro-2} respectively.

\begin{proposition}\label{pro:Explicit-Form} Under the assumptions of Theorem~\ref{thm:Main3D},
\begin{equation}\label{eq:Explicit-Form}
       \lim_{n\rightarrow\infty} \frac n{k!}\int_{\Omega}\left(n|B(x,r_n)\cap\Omega|\right)^k e^{-n|B(x,r_n)\cap\Omega|}dx= e^{-c}.
\end{equation}
\end{proposition}

\begin{proposition}\label{pro:Poisson-Version}
Under the assumptions of Theorem~\ref{thm:Main3D},
\begin{equation}\label{eq:concusion-1-Poisson}
\lim_{n\rightarrow\infty}\Pr\left\{\rho(\mathcal{P}_n;\delta\geq k+1)\leq r_n\right\} =\exp\left(-e^{-c}\right),
\end{equation}
where $\mathcal{P}_n$ is a homogeneous Poisson point process of intensity $n$ (i.e., $n|\Omega|$) distributed over unit-volume convex region $\Omega$.
\end{proposition}

\begin{proposition}\label{pro:conclusion-1}   Under the assumptions of Theorem~\ref{thm:Main3D},
\begin{equation}\label{eq:conclusion-1}
\lim_{n\rightarrow\infty}\Pr\left\{\rho(\chi_n;\delta\geq k+1)\leq r_n\right\} =e^{-e^{-c}}.
\end{equation}
\end{proposition}

\begin{proposition}\label{pro:conclusion-2}  Under the assumptions of Theorem~\ref{thm:Main3D},
\begin{equation}\label{eq:conclusion-2}
\lim_{n\rightarrow\infty}\Pr\left\{\rho(\chi_n;\delta\geq k+1)=\rho(\chi_n;\kappa\geq
k+1)\right\}=1.
\end{equation}
\end{proposition}

\par
We use the following notations throughout this article.
(1) Region $\Omega \subset \mathbb{R}^3$ is a unit-volume convex  region, and $B(x,r)\subset \mathbb{R}^3$ is a ball
centered at $x$ with  radius $r$.
(2) Notation $|A|$ is a short  for the volume of a measurable set $A\subset \mathbb{R}^3$ and $\|\cdot\|$ represents the length of a line segment. $\mathrm{Area}(\cdot)$ denotes the area of a surface.
(3) $\mathrm{dist}(x, A)=\inf_{y\in A}\|xy\|$ where $x$ is a point and $A$ is a set.
(4) Given any two nonnegative functions $f(n)$ and $g(n)$, if there exist two constants $0<c_1<c_2$ such that
$c_1g(n)\leq f(n)\leq c_2g(n)$ for any sufficiently large $n$, then  denote $f(n)=\Theta(g(n))$.
We also use notations $f(n)=o(g(n))$ and $f(n)\sim g(n)$ to denote that $\lim\limits_{n\rightarrow\infty}\frac{f(n)}{g(n)}=0$ and $\lim\limits_{n\rightarrow\infty}\frac{f(n)}{g(n)}=1$, respectively. A surface is said to be smooth in this paper, meaning  that its function has continuous second derivatives.

\section{Proof of Proposition~\ref{pro:Explicit-Form}}\label{sec:Pro-1}

Throughout this article, we define
\begin{equation}\label{eq:psi-function}
       \psi^k_{n,r}(x)= \frac{\left(n|B(x,r)\cap\Omega|\right)^k
       e^{-n|B(x,r)\cap\Omega|}}{k!}.
\end{equation}
All left work in this section is to prove Proposition~\ref{pro:Explicit-Form}, i.e., $ n\int_{\Omega} \psi^k_{n,r}(x)dx\sim e^{-c}.$
The proof will follow the framework proposed in~\cite{Ding-RGG2018} to carefully deal with the boundary of a convex region.
The framework is developed  based on the pioneering work of Wan \emph{et al.}
in~\cite{PWan04:Asymptotic-critical-transmission-MobileHoc} and~\cite{PJWan-IT-asymptotic-radius}, although in which only regular
regions like disk or square are considered.

%
%The proof of the above part  of Theorem~\ref{thm:combined-Main} is divided into two cases. The first case is to assume the boundary $\partial\Omega$ to be smooth.
%In this case, region $\Omega$ is divided into four disjoint parts, i.e.,\ $\Omega=\Omega(0)\cup\Omega(2)\cup\Omega(1,2)\cup\Omega(1,1)$. The proof
%in this case will thus be completed within three steps: Step~1 deals with the integrals on $\Omega(0)$ and $\Omega(2)$, Step~2 copes with $\Omega(1,2)$
%while Step~3 with $\Omega(1,1)$. After that, we further treat the continuous boundary case through an approximation approach.

The following three conclusions are elementary, with their proof   presented in Appendix~A for reviewing.

\begin{lemma}\label{lemma:LowerBound-B(x,r)} Let $\Omega\subset \mathbb{R}^3$ be a bounded convex region,
then there exists a positive constant $C$ such that for any
sufficiently small $r$,
$
        \mathop{\inf}_{x\in \Omega}|B(x,r)\cap \Omega |\geq C\pi
        r^3.
$
In particular, if $\partial \Omega$ is smooth, then $\forall x\in \partial \Omega$,
$\lim_{r\rightarrow 0}\frac{|B(x,r)\cap \Omega |}{\frac{4}{3}\pi r^3}=\frac{1}{2}$.
\end{lemma}

%
%\begin{figure}[t]
%\begin{center}
%\subfigure{
%\begin{minipage}{5cm}
%\includegraphics[scale=0.36] {figures/Figure5-1}
% \caption{Illustration of Lemma~\ref{lem:shadow-low-bound}: area $S$ }
%\label{fig:Lemma2-Figure5-1}
%\end{minipage}
%}
%\subfigure{
%\begin{minipage}{5cm}
%\includegraphics[scale=0.55]{figures/lemma2}
%\caption{$\partial \Omega$ is tangent to $CD$ at point $O$}\label{fig:Rectangle-ABCD}
%\end{minipage}
%}
%\end{center}
%\end{figure}

%
%
 \begin{figure}[htbp]
 \begin{center}
       \scalebox{0.50}[0.50]{
          \includegraphics{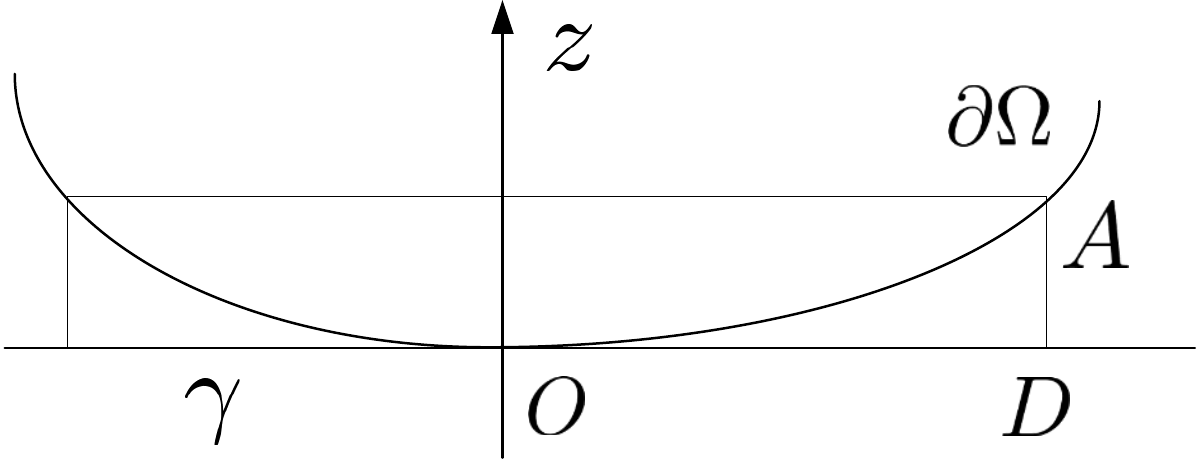}
       }\caption{$\partial \Omega$ is tangent to $\gamma$ at point $O$}\label{fig:lemma2}
  \end{center}
  \end{figure}

\begin{lemma}\label{lemma:G(Omega)}
Suppose smooth surface $\partial \Omega$ is tangent to   plane  $\gamma$  at point $O$, seeing Figure~\ref{fig:lemma2}. Point $A\in\partial \Omega$ and $D\in\gamma$, $AD \perp\gamma$ at point $D$. Then
\begin{equation}\label{eq:lem-G(Omega)-1}
 \|AD\|\leq G(\Omega)\|OD\|^2+o(\|OD\|),
\end{equation}
where  $G(\Omega)>0$ is a constant only depending on $\Omega$.
This leads to
that if $\|AD\|> (G(\Omega)+1)r^2$, then
$\|OD\|> r, $
as long as $r$ is sufficiently small.
\end{lemma}

For any $t\in[0,r]$, we
define
\begin{equation}
  a(r,t)=|\{x=(x_1,x_2,x_3):x_1^2+x_2^2+x_3^2\leq r^2, x_1\leq t\}|.
\end{equation}
$ a(r,t)$ is usually shortly denoted by $a(t)$. Here this definition follows the ones in~\cite{PWan04:Asymptotic-critical-transmission-MobileHoc, PJWan-IT-asymptotic-radius} where only the sets on planes are considered.
\begin{lemma}\label{lemma:a(t)}
%The following lemma is first given by Wan \emph{et al.} in~\cite{PWan04:Asymptotic-critical-transmission-MobileHoc}.
Let $r=r_n=\left(\frac{16}{5\pi}\frac{\log n+(\frac{3k}{2}-1)\log\log n+\xi}{ n }\right)^{\frac13}$ and $k\geq 1$, then
$$
    n\int_0^{\frac r2}\frac{(na(t))^ke^{-na(t)}}{k!}dt\sim
   \frac{4}{3\pi}e^{-\frac{2\xi}{3}}\left(\frac{5\pi}{16}\right)^{\frac23}\left(\frac{2}{3}\right)^k\frac{1}{k!}.
$$
\end{lemma}

{
%The following estimation about $r_n$ given in the theorem is straight forward, and will be often used.  We list them as a remark.
%\begin{remark}\label{rem:remark3}
%$r_n$ given in Theorem~\ref{thm:Main3D} satisfies
%\begin{enumerate}
%
%   \item $ {n\pi r^3}\rightarrow\infty$ as $n\rightarrow\infty$.
%
%  \item When $k\geq 1$,
% $
%     \frac n{k!}(n\pi r^3)^k e^{-n\pi r^3}=o(1).
%$
%  \item  If $C>0$ is a constant, then
% $$
%     (n\pi r^3)^{k+1}\exp(-Cn\pi r^3)=o(1).
%$$
%\end{enumerate}
%\end{remark}

\subsection{Case of smooth $\partial \Omega$}
In this subsection, we assume the boundary $\partial \Omega$ be smooth.
Let
$$
     \Omega(0)=\left\{x\in\Omega:\mathrm{dist}(x,\partial\Omega)\geq r\right\}, \quad
     \Omega(2)=\left\{x\in\Omega:\mathrm{dist}(x,\partial\Omega)\leq
     (G(\Omega)+1)r^2\right\},
$$
and define $
     \Omega(1)=\Omega\backslash\left(\Omega(0)\cup\Omega(2)\right).
$
Here constant $G(\Omega)$ is given by Lemma~\ref{lemma:G(Omega)}, and $r$ is considered sufficiently small  so that
 $\Omega(0)$ and $\Omega(2)$ are disjoint.
Clearly,
\begin{equation*}
\begin{split}
&n\int_{\Omega}\psi^k_{n,r}(x)dx\mathrm{d}x=n\left(\int_{\Omega(0)}+\int_{\Omega(2)}+\int_{\Omega(1)}\right)\psi^k_{n,r}(x)dx\mathrm{d}x.
\end{split}
\end{equation*}
%The integrations   will be estimated
%in the following Claim~\ref{pro:Omega(0)}-\ref{pro:Omega(1,1)}.

\begin{claim}\label{pro:Omega(0)}
 $n\int_{\Omega(0)} \psi^k_{n,r}(x)dx\sim o(1). $
\end{claim}
\begin{proof}
$\forall x\in \Omega(0), |B(x,r)\cap\Omega|=\frac{4}{3}\pi r^3$.  Notice that $|\Omega(0)|\sim (1-\mathrm{Area}(\Omega)r)$.
\begin{eqnarray*}
     n\int_{\Omega(0)} \psi^k_{n,r}(x)dx =\frac n{k!}\left(\frac{4n\pi r^3}{3}\right)^k e^{-\frac{4n\pi r^3}{3}}|\Omega(0)|
      \sim  \frac n{k!}\left(\frac{4n\pi r^3}{3}\right)^k e^{-\frac{4n\pi r^3}{3}} = o(1).
\end{eqnarray*}
\end{proof}

{ \begin{claim}\label{pro:Omega(2)}
$
   n\int_{\Omega(2)} \psi^k_{n,r}(x)dx=o(1).
$
\end{claim}}
\begin{proof} Notice that  $|\Omega(2)|\leq \mathrm{Area}(\Omega)\times (G(\Omega)+1)r^2=\Theta(1)r^2
=\Theta\left(\left(\frac{\log n}{n}\right)^\frac{2}{3}\right)$. By Lemma~\ref{lemma:LowerBound-B(x,r)}, there
exists $r_0$ such that $\forall r<r_0$, $|B(x,r)\cap \Omega | > \frac{1}{2}\pi r^3 $, which implies
$\sup_{x\in\Omega(2)} e^{-n|B(x,r)\cap \Omega |}=o(n^{-1})$.

\begin{eqnarray*}
   n\int_{\Omega(2)} \psi^k_{n,r}(x)dx&\leq &\frac n{k!}(n\pi r^3)^k[o(n^{-1})]|\Omega(2)|=o(1).
\end{eqnarray*}
\end{proof}

\begin{figure}[t]
\begin{center}
\subfigure{
\begin{minipage}{6cm}
\includegraphics[scale=0.65] {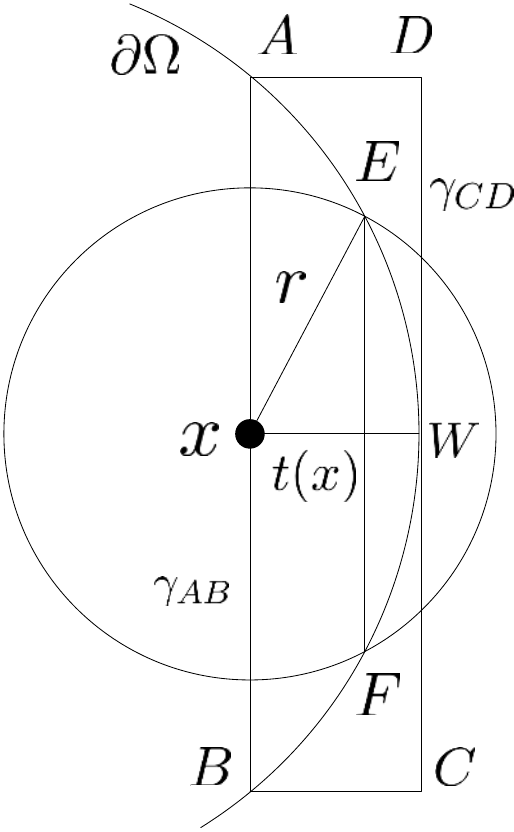}
  \caption{$t(x)$ is nonnegative}\label{positive-at)}
\end{minipage}
}
\subfigure{
\begin{minipage}{6cm}
\includegraphics[scale=0.65]{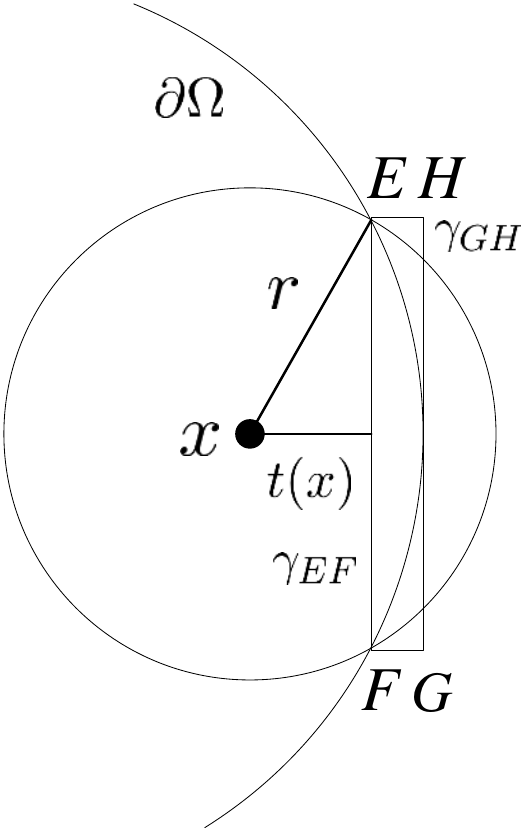}
\caption{Upper bound for $|B(x,r)\cap
       \Omega|$}\label{fig:UpperBound-Ball-Omega}
\end{minipage}
}
\end{center}
\end{figure}

%\begin{figure}[htbp]
% \begin{center}
%       \scalebox{0.7}[0.7]{
%           \includegraphics{figures/positive-a(t)}       }
%       \caption{$t(x)$ is nonnegative}\label{positive-a(t)}
%  \end{center}
%  \end{figure}
%
%
%\begin{figure}[htbp]
% \begin{center}
%       \scalebox{0.7}[0.7]{
%           \includegraphics{figures/thm2}       }
%       \caption{Upper bound for $|B(x,r)\cap
%       \Omega|$}\label{fig:UpperBound-B(x,r)Omega}
%  \end{center}
%  \end{figure}

For any point $x\in \Omega(1)$, let $W\in\partial \Omega$ be the nearest point to $x$, i.e.,
$
  \|xW\|=\mathrm{dist}(x,\partial \Omega)<r.
$
Let plane $\gamma_{CD}$ be tangent to $\partial \Omega$ at point $W$, plane $\gamma_{AB}$  parallel $\gamma_{CD}$ and pass through point $x$. Then $xW\bot\gamma_{CD}$ and $xW\bot\gamma_{AB}$. See Figure~\ref{positive-at)}.
 The distance between these two
planes is $\|xW\|>(G(\Omega)+1)r^2$, then by  Lemma~\ref{lemma:G(Omega)}  we have that
$$
   \mathrm{dist}(x, \partial \Omega \cap \gamma_{AB})>r,
$$
which implies that $\forall A\in\partial\Omega \cap \gamma_{AB}$, $$\|xA\|>r.$$
%Notice that
%$$
%    \|xW\|=\mathrm{dist}(x,\partial \Omega)<r.
%$$
The distance
between point $x$ and   point $y\in\partial \Omega$ is a continuous
function of $y$, which is due to the smooth
boundary $\partial \Omega$. By $\|xW\|<r$ and
$\|xA\|>r$, we know that any point $E'\in \partial \Omega$ with $\|xE'\|=r$, or any  point $E'\in \partial \Omega\cap \partial B(x,r)$, is between these two planes. Define
$$
    t(x,r)=\inf\{\mathrm{dist}(x,E'F')\mid E',F'\in\partial\Omega\cap\partial B(x,r)\},
$$
then clearly $t(x,r)\geq 0$.  Here $t(x,r)$ is usually shortly denoted by $t(x)$. Because $\partial\Omega\cap\partial B(x,r)$ is compact, we assume $E, F\in\partial\Omega\cap\partial B(x,r)$ such that
$$t(x,r)=\mathrm{dist}(x,EF)=\inf\{\mathrm{dist}(x,E'F')\mid E',F'\in\partial\Omega\cap\partial B(x,r)\}.$$

Furthermore, we  set
$$
    \Omega(1,1)=\left\{x\in \Omega(1):t(x)\leq \frac r2\right\},
\quad
   \Omega(1,2)=\Omega(1)\setminus\Omega(1,1).
$$
%then
%\begin{equation*}
%n\int_{\Omega(1)}\psi^k_{n,r}(x)dx=n\left(\int_{\Omega(1,1)}+\int_{\Omega(1,2)}\right)\psi^k_{n,r}(x)dx.
%\end{equation*}
In following, we will specify $n\int_{\Omega(1,2)}\psi^k_{n,r}(x)dx$ in
Claim~\ref{pro:Omega(1,2)}, and then  determine $n\int_{\Omega(1,1)}\psi^k_{n,r}(x)dx$ in
Claim~\ref{pro:Omega(1,1)}.
But at first, we will establish a lower and upper bounds for $|B(x,r)\cap\Omega|$ where $x\in\Omega(1)$.

\par
Because $t(x)>0$,
a lower bound is straightforward:
$$
|B(x,r)\cap\Omega|\geq a(t(x))\geq \frac{2}{3}\pi r^3.
$$
Now we give an upper bound for $|B(x,r)\cap \Omega|$.
Let $\gamma_{EF}$ be the plane passing through $EF$ with $\mathrm{dist}(x, \gamma_{EF})=t(x)$.
Plane $\gamma_{GH}$ parallels $\gamma_{EF}$ and tangents to $\partial\Omega$. See Figure~\ref{fig:UpperBound-Ball-Omega}.
 Here we assume $EH\bot\gamma_{GH}$ and
$FG\bot\gamma_{GH}$.

Clearly $\|EF\|<2r$, so $$\mathrm{Area}(B(x,r)\cap\gamma_{EF})\leq \Theta(1) r^2,$$
 and by Lemma~\ref{lemma:G(Omega)}, $\|EH\|<\Theta(1)r^2$. Therefore, the volume of
cylinder $\mathrm{Cyc}(EFGH)$ which is  formed by  surface $ B(x,r) \cap\gamma_{EF}$ and its projection on $\gamma_{FG}$}, satisfies
 $$\mathrm{Vol}(\mathrm{Cyc}(EFGH))=\mathrm{Area}(B(x,r)\cap\gamma_{EF})\|EH\|\leq \Theta (1)r^4.$$
  So
 \begin{eqnarray*}
     &&|B(x,r)\cap \Omega|\leq a(t(x))+\mathrm{Vol}(\mathrm{Cyc}(EFGH))\leq a(t(x))+\Theta(r^4)\leq a(t(x))+\Theta(r^4).
\end{eqnarray*}
With $a(t(x))\geq \frac {2\pi r^3}3$ that has been derived above, we obtain
\begin{eqnarray*}
     &&|B(x,r)\cap \Omega|\leq
     a(t(x))\left(1+\frac{\Theta(r^4)}{a(t(x))}\right)\leq a(t(x))(1+o(1)).
\end{eqnarray*}
According to these  lower and upper bounds of $|B(x,r)\cap \Omega|$, we have the
following estimates:
\begin{equation}\label{eq:local-1}
\begin{split}
     &(n|B(x,r)\cap \Omega|)^ke^{-n|B(x,r)\cap \Omega|}\leq (1+o(1))^k(na(t(x)))^ke^{-na(t(x))}.
\end{split}
\end{equation}
\begin{equation}\label{eq:local-2}
\begin{split}
     & (n|B(x,r)\cap \Omega|)^ke^{-n|B(x,r)\cap \Omega|}\geq (na(t(x)))^ke^{-n(a(t(x))+\Theta(r^4))}\geq  e^{-n \Theta(r^4)}(na(t(x)))^ke^{-na(t(x))}.
\end{split}
\end{equation}
\par Consider the integration on $\Omega(1,2)$. Noticing that volume
$|\Omega(1,2)|\leq \mathrm{Area}(\Omega)r$ and $\frac23 \pi r^3\leq a(t(x))\leq \frac{4}{3}\pi r^3$, then by formula~(\ref{eq:local-1}), we have
\begin{eqnarray*}
     n\int_{\Omega(1,2)} \psi^k_{n,r}(x)dx&=&\frac n{k!}\int_{\Omega(1,2)}\left(n|B(x,r_n)\cap\Omega|\right)^k e^{-n|B(x,r_n)\cap\Omega|}dx\\
      &\leq&\frac n{k!}\left[(1+o(1))^k(na(t(x)))^ke^{-na(t(x))}\right]|\Omega(1,2)|\\
     &\leq &\frac n{k!}(\frac{4}{3}n\pi r^3)^k e^{-\frac {2}{3}n\pi r^3}|\Omega(1,2)|=o(1).
\end{eqnarray*}
This proves
\begin{claim}\label{pro:Omega(1,2)}
$
n\int_{\Omega(1,2)} \psi^k_{n,r}(x)dx=o(1).
$
\end{claim}

\begin{figure}[htbp]
 \begin{center}
       \scalebox{0.55}[0.6]{
           \includegraphics{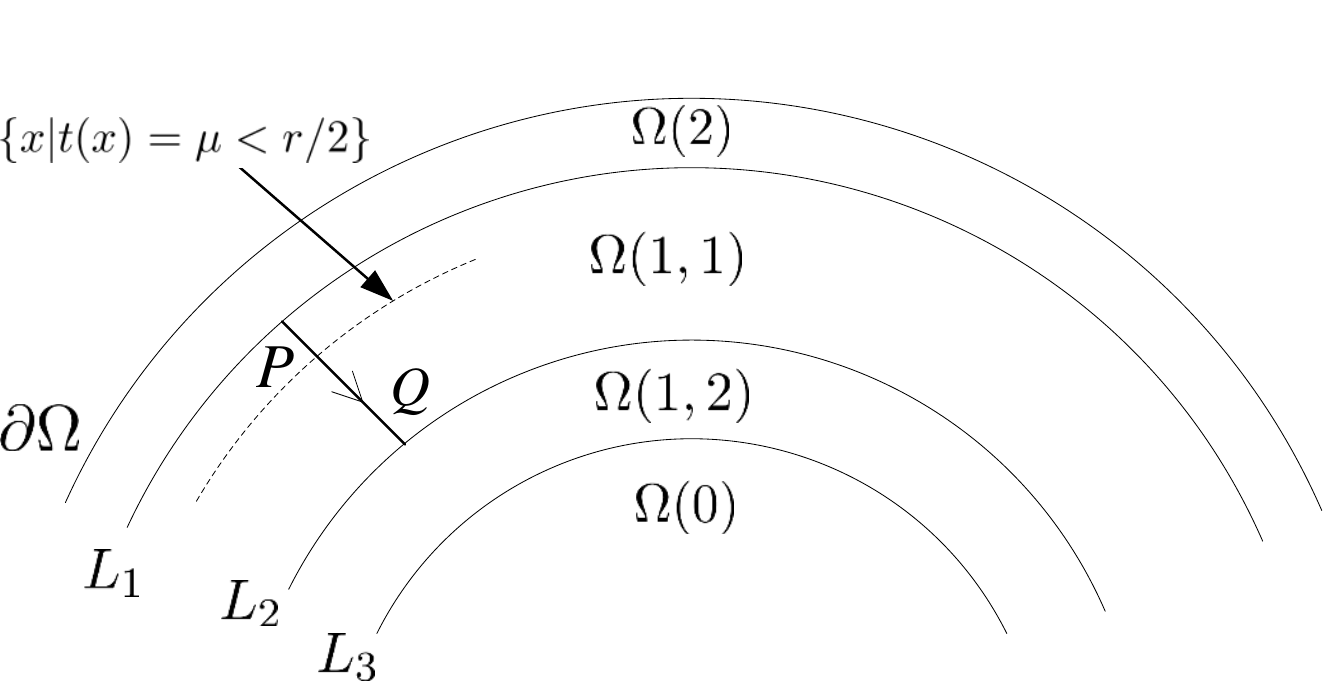}       }
       \caption{$\Omega$ is divided into four parts}\label{fig:Omega(1)}
  \end{center}
  \end{figure}

In the following, we will determine $n\int_{\Omega(1,1)} \psi^k_{n,r}(x)dx$.
By
estimates (\ref{eq:local-1}) and (\ref{eq:local-2}),
\begin{eqnarray*}
     &&\frac n{k!}\int_{\Omega(1,1)}\left(n|B(x,r_n)\cap\Omega|\right)^k e^{-n|B(x,r_n)\cap\Omega|}dx\sim \frac
     n{k!}\int_{\Omega(1,1)}(na(t(x)))^ke^{-na(t(x))}dx.
\end{eqnarray*}

We define
\begin{equation*}\label{eq:L1}
L_1=\{x\in\Omega\mid \mathrm{dist}(x,\partial \Omega)=(G(\Omega)+1)r^2\},
\quad
L_2=\{x\in\Omega(1)\mid t(x)=\frac{r}{2}\},
\end{equation*}
\begin{equation*}\label{eq:L3}
L_3=\{x\in\Omega\mid \mathrm{dist}(x,\partial \Omega)=r\}.
\end{equation*}
Clearly, the subregion $\Omega(1,1)$ has boundaries
 $L_1$ and $L_2$,   illustrated by
Figure~\ref{fig:Omega(1)}. If $x\in L_1$, then $t(x)\geq 0$ by a similar argument as shown in the previous
step. Clearly, $t(x)\leq (G(\Omega)+1)r^2$.
 Therefore, $0\leq t(x)\leq (G(\Omega)+1)r^2$ for  any $x\in L_1$. As $r$
 tends to zero, both surfaces $L_1$ and $L_2$ approximate the
 boundary $\partial \Omega$. So there exists $\epsilon_r$ such
 that  the area of $L_1$ and $L_2$,
 denoted by $\mathrm{Area}(L_1)$ and $\mathrm{Area}(L_2)$ respectively, satisfies:
$$
      \mathrm{Area}(\partial\Omega)-\epsilon_r\leq \mathrm{Area}(L_1),\mathrm{Area}(L_2)\leq \mathrm{Area}(\partial\Omega),
$$
where $\epsilon_r\rightarrow 0$ as $r\rightarrow 0$. In addition, it
is clear that $t(x)$ is increasing along the directed line segment
started from a point in $L_1$ to a point in $L_2$. See the line
segment $PQ$ in Figure~\ref{fig:Omega(1)}.

The integration on   $\Omega(1,1)$  can be bounded as
follows:
\begin{equation}\label{eq:Bounds-Integration-Omega(1,1)}
\begin{split}
&(\mathrm{Area}(\Omega)-\epsilon_r)\frac n{k!}\int_{(G+1)r^2}^{\frac r2}(na(t))^ke^{-na(t)}dt\\
&\leq \frac n{k!}
\int_{\Omega(1,1)}(na(t(x)))^ke^{-na(t(x))}dx\\
&\leq \mathrm{Area}(\Omega)\frac n{k!}\int_0^{\frac r2}(na(t))^ke^{-na(t)}dt.
\end{split}
\end{equation}

According to Lemma~\ref{lemma:a(t)}, we obtain that
\begin{eqnarray*}
    % &&\int_{\Omega(1,1)}(na(t(x)))^ke^{-na(t(x))}dx\\
  \mathrm{Area}(\Omega)\frac n{k!}\int_0^{\frac r2}(na(t))^ke^{-na(t)} dt\sim\mathrm{Area}(\Omega)\frac{4}{3\pi}e^{-\frac{2\xi}{3}}\left(\frac{5\pi}{16}\right)^{\frac23}\left(\frac{2}{3}\right)^k\frac{1}{k!}. \end{eqnarray*}
Therefore,
\begin{equation*}
\begin{split}
      0&\leq\epsilon_r\int_{(G+1)r^2}^{\frac{r}{2}}\frac
      n{k!}(na(t))^ke^{-na(t)}dt\leq\epsilon_r\int_{0}^{\frac{r}{2}}\frac
      n{k!}(na(t))^ke^{-na(t)}dt=o(1).
\end{split}
\end{equation*}
Furthermore, noticing that $2\pi r^3/3\leq a(t)\leq 2(\pi+\delta)r^3/3$
for any $t\in [0,(G(\Omega)+1)r^2]$, where $\delta>0$ is a small real number, we
have that
$$
      \mathrm{Area}(\Omega)\int_0^{(G+1)r^2}\frac n{k!}(na(t))^ke^{-na(t)}dt=o(1).
$$
So,
\begin{equation*}
\begin{split}
&( \mathrm{Area}(\Omega)-\epsilon_r)\frac n{k!}\int_{(G+1)r^2}^{\frac r2}(na(t))^ke^{-na(t)}dt\\
=&\frac
{n}{k!}\left( \mathrm{Area}(\Omega)\int_0^{\frac{r}{2}}- \mathrm{Area}(\Omega)\int_0^{(G+1)r^2}-\epsilon_r\int_{(G+1)r^2}^{\frac
{r}{2}}\right)(na(t))^ke^{-na(t)}dt\\
=& \mathrm{Area}(\Omega)\frac {n}{k!}\int_0^{\frac{r}{2}}(na(t))^ke^{-na(t)}dt+o(1).
\end{split}
\end{equation*}
Then by formula~(\ref{eq:Bounds-Integration-Omega(1,1)}),
\begin{eqnarray*}
     && \mathrm{Area}(\Omega)\frac {n}{k!}\int_0^{\frac{r}{2}}(na(t))^ke^{-na(t)}dt+o(1)\\
     &=&( \mathrm{Area}(\Omega)-\epsilon_r)\frac n{k!}\int_{(G+1)r^2}^{\frac r2}(na(t))^ke^{-na(t)}dt\\
     &\leq&\frac n{k!}\int_{\Omega(1,1)}(na(t(x)))^ke^{-na(t(x))}dx\\
     &\leq&  \mathrm{Area}(\Omega)\frac n{k!}\int_0^{\frac r2}(na(t))^ke^{-na(t)}dt,
\end{eqnarray*}
which implies that
\begin{eqnarray*}
&&\frac n{k!}\int_{\Omega(1,1)}(na(t(x)))^ke^{-na(t(x))}dx \sim  \mathrm{Area}(\Omega)\frac n{k!}\int_0^{\frac r2}(na(t))^ke^{-na(t)} dt.
\end{eqnarray*}
So, by Lemma~\ref{lemma:a(t)}, we have
\begin{eqnarray*}
     &&\frac n{k!}\int_{\Omega(1,1)}\left(n|B(x,r_n)\cap\Omega|\right)^k e^{-n|B(x,r_n)\cap\Omega|}dx\\
     &\sim &\frac{n}{k!}\int_{\Omega(1,1)}(na(t(x)))^ke^{-na(t(x))}dx\\
     &\sim&  \mathrm{Area}(\Omega)\frac {n}{k!}\int_0^{\frac r2}(na(t))^ke^{-na(t)}dt\\
     &\sim&\mathrm{Area}(\Omega)\frac{4}{3\pi}e^{-\frac{2\xi}{3}}\left(\frac{5\pi}{16}\right)^{\frac23}\left(\frac{2}{3}\right)^k\frac{1}{k!}.
\end{eqnarray*}
Therefore, we have proved the following conclusion:
\begin{claim}\label{pro:Omega(1,1)}
$
     n\int_{\Omega(1,1)} \psi^k_{n,r}(x)dx\sim\mathrm{Area}(\Omega)\frac{4}{3\pi}e^{-\frac{2\xi}{3}}\left(\frac{5\pi}{16}\right)^{\frac23}\left(\frac{2}{3}\right)^k\frac{1}{k!}.
$
\end{claim}
The four claims prove  $\int_{\Omega}n\psi^k_{n,r}(x)dx=\left\{\int_{\Omega(0)}+\int_{\Omega(2)}+\int_{\Omega(1,2)}+\int_{\Omega(1,1)}\right\}
    n\psi^k_{n,r}(x)dx\sim e^{-c}.$
%\begin{eqnarray*}
%   &&\int_{\Omega}n\psi^k_{n,r}(x)dx=\left\{\int_{\Omega(0)}+\int_{\Omega(2)}+\int_{\Omega(1,2)}+\int_{\Omega(1,1)}\right\}
%    n\psi^k_{n,r}(x)dx\sim e^{-c}.
%\end{eqnarray*}

\subsection{Case of continuous $\partial \Omega$}

For a general unit-volume convex region $\Omega$ with a continuous
rather than smooth boundary $\partial \Omega$, we may use a family
of convex regions $\{\Omega_n\}_{n=1}^{\infty}\subset \Omega$ to approximate
$\Omega$, where the boundary $\partial \Omega_n$ of each $\Omega_n$
is smooth. We set the width of the gap between $\partial \Omega$ and
$\partial \Omega_n$ to be less than $r^2_n$,
%see Figure~\ref{fig:Gap}.
That is, $\sup_{x\in \Omega_n}\mathrm{dist}(x,\partial
\Omega)<r^2_n$ for any $n$. Clearly, $\mathrm{Vol}(\Omega_n)$, the volume of $\Omega_n$, satisfies
$
    1-\mathrm{Area}(\Omega)r_n^2\leq \mathrm{Vol}(\Omega_n)\leq 1.
$
%\begin{figure}[htbp]
% \begin{center}
%       \scalebox{0.7}[0.7]{
%           \includegraphics{figures/Gap}       }
%       \caption{Width of the Gap between $\partial \Omega$ and $\partial \Omega_n$ is less than $r_n^2$}\label{fig:Gap}
%  \end{center}
%  \end{figure}
%
 Because $\partial
\Omega_n$ is smooth, follow the  method presented in the previous
subsection, we can also similarly obtain $
n\int_{\Omega_n}\psi^k_{n,r}(x)dx\sim e^{-c}.
$
Since the volume of $\Omega\setminus \Omega_n$ is no more than $\mathrm{Area}(\Omega)r_n^2$,
then by the proof of Claim~\ref{pro:Omega(2)}, we have that
$
n\int_{\Omega\setminus \Omega_n}\psi^k_{n,r}(x)dx=o(1).
$
Therefore,
$$
n\int_{\Omega}\psi^k_{n,r}(x)dx=n\left(\int_{\Omega_n}+\int_{\Omega\setminus
\Omega_n}\right)\psi^k_{n,r}(x)dx\sim e^{-c}.
$$
This finally  completes the proof of Proposition~\ref{pro:Explicit-Form}.

\section{Proof of Proposition~\ref{pro:Poisson-Version}}\label{sec:Pro-2}
We follow Penrose's approach and framework to prove Proposition~\ref{pro:Poisson-Version}.
 For given $n$, $x,y\in \Omega$, let
  $$
   v_x=|B(x,r)\cap \Omega |, v_y=|B(y,r)\cap \Omega |,  v_{x,y}=|B(x,r)\cap B(y,r)\cap \Omega|,$$
   $$v_{x\backslash y}=v_x-v_{x,y},v_{y\backslash x}=v_y-v_{x,y}.
  $$
Define $I_i=I_i(n)(i=1,2,3)$ as follows:
$$
    I_1=n^2\int_{\Omega} dx\int_{\Omega\cap B(x,3r)}dy \psi^k_{n,r}(y)\psi^{k}_{n,r}(x),
$$
$$
   I_2=n^2\int_{\Omega} dx\int_{\Omega\cap B(x, r)}dy \Pr[Z_1+Z_2=Z_1+Z_3=k-1],
$$
$$
   I_3=n^2\int_{\Omega} dx\int_{\Omega\cap B(x,3r)\backslash B(x,r)}dy \Pr[Z_1+Z_2=Z_1+Z_3=k],
$$
where $Z_1,Z_2,Z_3$ are independent Poisson variables with means
$nv_{x,y},nv_{x\backslash y},nv_{y\backslash x}$ respectively. As Penrose has pointed out in~\cite{Penrose-k-connectivity},
by an argument similar to that of Section~7 of~\cite{Penrose97:longest-edge}, to prove Proposition~\ref{pro:Poisson-Version} it suffices to prove that
$I_1,I_2,I_3\rightarrow 0$ as $n\rightarrow \infty$. Firstly, we give  the following conclusion.

\begin{proposition}\label{proposition1} Under the assumptions of Theorem~\ref{thm:Main3D}, $Z_3$ is a Poisson variable with mean $nv_{y\backslash x}$ defined above, then

  \begin{equation*}
\frac{1}{n\pi r^3}\int_{\Omega}n\psi^k_{n,r}(x)dx\int_{\Omega\cap
B(x,r)}n\Pr(Z_3=k-1)dy=o(1).
\end{equation*}
\end{proposition}
\begin{proof}
The proof is shown in Appendix~B.
\end{proof}

The conclusion of $I_1,I_2,I_3\rightarrow 0$, as
$n\rightarrow 0$, will be proved by the following three claims.

\begin{claim}
$I_1\rightarrow 0$ as $n\rightarrow \infty$.
\end{claim}
\begin{proof}
By Lemma~\ref{lemma:LowerBound-B(x,r)}, there exists a constant $C>0$  such that $C\pi r^3\leq |B(x,r)\cap\Omega|\leq \frac{4}{3}\pi r^3$.
 \begin{eqnarray*}
 I_1
&=&n^2\int_{\Omega} dx\int_{{\Omega}\cap B(x,3r)}dy \psi^k_{n,r}(y)\psi^{k}_{n,r}(x)\\
    &=& \int_{\Omega} n\psi^{k}_{n,r}(x)dx \int_{B(x,3r)\cap \Omega}n\psi^{k}_{n,r}(y)dy\\
    &\leq& \Theta(1)\int_{\Omega} n\psi^{k}_{n,r}(x)dx \int_{B(x,3r)\cap \Omega}\frac{1}{k!} n(n\pi r^3)^k\exp(-Cn\pi r^3)\\
    &\leq & \Theta(1)\frac{n}{k!}|B(x,3r)|(n\pi r^3)^{k}\exp(-Cn\pi r^3)\int_{\Omega} n\psi^{k}_{n,r}(x)dx\\
    &=& \Theta(1)(n\pi r^3)^{k+1}\exp(-Cn\pi r^3)\int_{\Omega} n\psi^{k}_{n,r}(x)dx\\
    &\sim& \Theta(1)(n\pi r^2)^{k+1}\exp(-Cn\pi r^3)=o(1).
\end{eqnarray*}
\end{proof}

\begin{claim}
$I_2\rightarrow 0$, as $n\rightarrow \infty$.
\end{claim}
\begin{proof}
  It is obvious that
$Z_1+Z_2$ and $Z_3$ are Possison variables with means $nv_x$ and
$nv_{y\backslash x}$ respectively.
Notice that
\begin{eqnarray*}
&&\Pr[Z_1+Z_3=k-1| Z_1+Z_2=k-1]\\
%&=& \sum_{j=0}^{k-1}\Pr[Z_1+Z_3=k-1;Z_1=j|Z_1+Z_2=k-1]\\
&\leq &\sum_{j=0}^{k-1}\Pr[Z_3=k-1-j |Z_1+Z_2=k-1]\\
&=&\sum_{j=0}^{k-1}\Pr[Z_3=k-1-j].
\end{eqnarray*}

For any $x\in\Omega$, $\psi^{k}_{n,r}(x)=\frac{n|B(x,r)\cap \Omega|}{k}\psi^{k-1}_{n,r}(x),$
then by Lemma~\ref{lemma:LowerBound-B(x,r)}, there exists a constant $C>0$ such that
$C\pi r^3\leq |B(x,r)\cap \Omega|$, and thus
$ \psi^{k-1}_{n,r}(x)\leq \frac{k}{Cn\pi r^3}\psi^{k}_{n,r}(x).$
Therefore,
\begin{eqnarray*}
    I_2
    &=&n^2\int_{\Omega} dx\int_{\Omega\cap B(x,r)}dy \Pr[Z_1+Z_2=Z_1+Z_3=k-1]\\
       &=&n^2\int_{\Omega} dx\int_{\Omega\cap B(x,r)}dy
       \Pr[Z_1+Z_2=k-1]  \Pr[Z_1+Z_3=k-1|Z_1+Z_2=k-1]\\
       &=& n \int_{\Omega}\Pr[Z_1+Z_2=k-1] dx \int_{\Omega\cap B(x,r)}n\Pr[Z_1+Z_3=k-1|Z_1+Z_2=k-1]dy\\
      % &=& n \int_{\Omega}P[Z_1+Z_2=k-1] dx\\
      % && \left\{\sum_{j=0}^{k-1}\int_{\Omega\cap B(x,r)}nP[Z_1+Z_3=k-1;Z_1=j|Z_1+Z_2=k-1]dy\right\}\\
      % &=& n \int_{\Omega}P[Z_1+Z_2=k-1] dx\\
      % && \left\{\sum_{j=0}^{k-1}\int_{\Omega\cap B(x,r)}nP[Z_3=k-1-j|Z_1+Z_2=k-1]dy\right\}\\
       &\leq & n\int_{\Omega}\Pr[Z_1+Z_2=k-1] dx \left\{\sum_{j=0}^{k-1}\int_{\Omega\cap B(x,r)}n\Pr[Z_3=k-1-j]dy\right\}\\
      % &=& \sum_{j=0}^{k-1}n \int_{\Omega}P[Z_1+Z_2=k-1] dx \int_{\Omega\cap B(x,r)}nP[Z_3=k-1-j]dy\\
       &=& \sum_{j=0}^{k-1} \int_{\Omega}n\psi^{k-1}_{n,r}(x) dx \int_{\Omega\cap B(x,r)}n\Pr[Z_3=k-1-j]dy\\
       &\leq& \frac{C}{k} \sum_{j=0}^{k-1}
       \frac{1}{n\pi r^3}\int_{\Omega}n\psi^{k}_{n,r}(x) dx \int_{\Omega\cap B(x,r)}n\Pr[Z_3=k-1-j]dy=o(1).
\end{eqnarray*}
The last equation holds due to Proposition~\ref{proposition1} and the proved conclusion $\int_{\Omega}n\psi^{k}_{n,r}(x) dx \sim
e^{-c}$.

\end{proof}

Similarly, we can prove
\begin{claim}
$I_3\rightarrow 0$, as $n\rightarrow \infty$.
\end{claim}

Therefore, by these three claims,  we know the Poissonized version (\ref{eq:concusion-1-Poisson}) holds. The proof of Proposition~\ref{pro:Poisson-Version} is completed.

\section{Proofs of Proposition~\ref{pro:conclusion-1} and~\ref{pro:conclusion-2} }

%We will follow Penrose's approach to prove these two conclusions. This section mainly presents a proof for Conclusion~1, while we
%sketch the proof of Conclusion~2 in Appendix~C. The approach  is instead to prove  a Poissonized version of Conclusion~1:
%\begin{equation}\label{eq:concusion-1-Poisson}
%\lim_{n\rightarrow\infty}\Pr\left\{\rho(\mathcal{P}_n;\delta\geq k+1)\leq r_n\right\} =\exp\left(-e^{-c}\right),
%\end{equation}
%where $\mathcal{P}_n$ means that the $n$ points distributed over unit-area convex region $\Omega$ are Poisson point process, which
%has been introduced in Section~\ref{sec-related-mainresults}.  Because $\chi_n$ can be well approximated by $\mathcal{P}_n$,
%it is naturally  to see the original (\ref{eq:conclusion-1}) hold by a de-Poissonized technique.
%The de-Poissonized technique is standard and thus omitted here, please see~\cite{Penrose-k-connectivity} for details. In the following, we only demonstrate
%the Poissonized version (\ref{eq:concusion-1-Poisson}) holds for a general convex region $\Omega$.

Proposition~\ref{pro:Poisson-Version} can lead to Proposition~\ref{pro:conclusion-1} by a de-Poissonized technique.  The de-Poissonized technique is standard and thus omitted here, please see~\cite{Penrose-k-connectivity} for details.

\par Here we sketch the proof of Proposition~\ref{pro:conclusion-2}. Penrose has clearly proved this result when region $\Omega$ is a square~\cite{Penrose-k-connectivity}.  He  constructed two events, $E_n(K)$ and $F_n(K)$,  such that for any $K>0$,
$$
     \left\{\rho(\chi_n;\delta\geq k+1)\leq r_n<\rho(\chi_n;\kappa\geq
k+1)\right\}\subseteq E_n(K)\cup F_n(K),
$$
and
$$
\lim_{n\rightarrow \infty}\Pr[E_n(K)]=\lim_{n\rightarrow \infty}\Pr[F_n(K)]=0.
$$
The definition of the events and the convergence results are organised in
 Proposition~5.1 and~5.2 of~\cite{Penrose-k-connectivity}. We do not introduce them in detail. Please refer to~\cite{Penrose-k-connectivity}.
These conclusions can be  straightly generalised  to the case of
convex region, with
their proofs not much been modified. In fact, we have proved  Proposition~\ref{pro:Explicit-Form}:
 $n\int_{\Omega} \psi^k_{n,r}(x)dx\sim e^{-c}$, and
$|B(x,r)\cap\Omega|$ in  $$
       \psi^k_{n,r}(x)= \frac{\left(n|B(x,r)\cap\Omega|\right)^k
       e^{-n|B(x,r)\cap\Omega|}}{k!}
$$
can be bounded by $C\pi r^3\leq|B(x,r)\cap\Omega|\leq \frac{4}{3}\pi r^3$ where $C>0$, seeing Lemma~\ref{lemma:LowerBound-B(x,r)}.

As a result, based on   two generalised conclusions, a squeezing argument can lead to the following
$$\lim_{n\rightarrow\infty}\Pr\left\{\rho(\chi_n;\delta\geq k+1)=\rho(\chi_n;\kappa\geq
k+1)\right\}=1.$$
Please see the details presented in~\cite{Penrose-k-connectivity}.

\section*{References}
\bibliographystyle{elsarticle-num}
\bibliography{RGG3DFeb5}

\appendix

\section{Proofs of Lemma~\ref{lemma:LowerBound-B(x,r)}, Lemma~\ref{lemma:G(Omega)}  and Lemma~\ref{lemma:a(t)} }
The proofs of Lemma~\ref{lemma:LowerBound-B(x,r)} and Lemma~\ref{lemma:G(Omega)} are elementary and and similar to the ones presented in~\cite{Ding-RGG2018}. So they are omitted here. The following is the proof of Lemma~\ref{lemma:a(t)}.

\begin{proof} (Proof of Lemma~\ref{lemma:a(t)}). For convenience, we denote $C_k=\frac{3k}{2}-1$. Notice that
$$
    a(t)=\frac{\pi}{3}(r-t)^2(2r+t), $$
$$
   a'(t)=\pi(t^2-r^2), \quad  a''(t)=2\pi t.
$$
Let $f(t)=na(t)$, then
$$
   n\int_{0}^{\frac r2} \frac{(f(t))^ke^{-f(t)}}{k!}dt=-\frac{1}{a'(t)}e^{-f(t)}\sum_{k=0}^{k}\frac{(f(t))^i}{i!}|_{0}^{\frac r2}
   -\int_{0}^{\frac r2} \frac{a''(t)}{(a'(t))^2}e^{-f(t)}\sum_{i=0}^{k}\frac{(f(t))^i}{i!}dt.
$$
Notice that
$$
   f\left(\frac r2\right)=\frac{5n\pi r^3}{24}=\frac{5n\pi}{24}\frac{16}{5\pi}\frac{\log n+C_k\log\log n+\xi}{n}=\frac{2}{3}\left(\log n+C_k\log\log n+\xi\right).
$$
$$ e^{-f\left(\frac r2\right)}=\frac{1}{n^{\frac{2}{3}}}\frac{1}{(\log n)^{\frac{2}{3}C_k}}e^{-\frac{2\xi}{3}}.
$$

(i) First term:
$$
    -\frac{1}{a'\left(\frac{r}{2}\right)}e^{-f\left(\frac r2\right)}=\frac{1}{\frac{3}{4}\pi r^2}\frac{1}{n^{\frac{2}{3}}}\frac{1}{(\log n)^{\frac{2}{3}}}e^{-\frac{2\xi}{3}}=\frac{4}{3\pi}e^{-\frac{2\xi}{3}}\frac{\left(\frac{5\pi}{16}\frac{ n }{\log n+C_k\log\log n}\right)^{\frac23}}{n^{\frac{2}{3}}(\log n)^{\frac{2}{3}C_k}}
$$
$$
    -\frac{1}{a'\left(\frac{r}{2}\right)}e^{-f\left(\frac r2\right)}\sum_{i=0}^k\frac{f\left(\frac r2\right)^i}{i!}=\frac{4}{3\pi}e^{-\frac{2\xi}{3}}\frac{\left(\frac{5\pi}{16}\frac{ n }{\log n+ C_k\log\log n+\xi}\right)^{\frac23 }}{n^{\frac{2}{3}}(\log n)^{\frac{2}{3}C_k}}\sum_{i=0}^k\left(\frac{2}{3}\left(\log n+C_k\log\log n+\xi\right)\right)^i/i!
$$
Because $C_k=\frac{3}{2}k-1$, so $\left(\log n\right)^{\frac23+\frac23C_k}=\left(\log n\right)^k$, and
$$
    -\frac{1}{a'\left(\frac{r}{2}\right)}e^{-f\left(\frac r2\right)}\sum_{i=0}^k\frac{f\left(\frac r2\right)^i}{i!}
    \sim-\frac{1}{a'\left(\frac{r}{2}\right)}e^{-f\left(\frac r2\right)}\frac{f\left(\frac r2\right)^k}{k!}\sim
   \frac{4}{3\pi}e^{-\frac{2\xi}{3}}\left(\frac{5\pi}{16}\right)^{\frac23}\left(\frac{2}{3}\right)^k\frac{1}{k!}.
$$

(ii) Second term.  Notice that
$$
    a'(0)=-\pi r^2, \quad f\left(0\right)=\frac{32}{15}\left(\log n+C_k\log\log n+\xi\right), \quad  e^{-f(0)}=\frac{1}{n^{\frac{32}{15}}}\frac{1}{(\log n)^{\frac{32}{15}C_k}}e^{-\frac{32\xi}{15}}.
$$
It is easy to see that
$$
   -\frac{1}{a'\left(0\right)}e^{-f\left(0\right)}f\left(0\right)^k=o(1).
$$

(iii) Third term.
$$
   a'(t)=\pi(t^2-r^2), \quad  a''(t)=2\pi t, $$
 When $t\leq \frac{r}{2}$,
   $$
  \left|\frac{a''(t)}{(a'(t))^3}\right|=\frac{2}{\pi^2}\frac{t}{(r^2-t^2)^3}\leq \frac{64}{27\pi^2}\frac{1}{r^5}
  =\Theta\left(\left(\frac{n}{\log n+C_k\log\log n+\xi}\right)^{\frac{5}{3}}\right).
$$
Then
\begin{eqnarray*}
 &&\left|\int_{0}^{\frac r2} \frac{a''(t)}{(a'(t))^2}e^{-f(t)}\sum_{i=0}^{k}\frac{(f(t))^i}{i!}dt\right|\\
 &=&\frac{1}{n}\left|\int_{0}^{\frac r2} \frac{a''(t)}{(a'(t))^3}e^{-f(t)}\sum_{i=0}^{k}\frac{(f(t))^i}{i!}df(t)\right|\\
 &\leq & \frac{64}{27\pi^2}\frac{1}{nr^5} \left|\int_{0}^{\frac r2} e^{-f(t)}\sum_{i=0}^{k}\frac{(f(t))^i}{i!}df(t)\right|\\
 &\leq & \Theta(1)\frac{64}{27\pi^2}\frac{1}{nr^5} \left|\int_{0}^{\frac r2} e^{-f(t)}\frac{(f(t))^k}{k!}df(t)\right|\\
 &= & \Theta(1)\frac{64}{27\pi^2}\frac{1}{nr^5} \left|\int_{0}^{\frac r2} d\left(-e^{-f(t)}\sum_{i=0}^{k}\frac{(f(t))^i}{i!}\right) \right|\\
  &\leq & \Theta(1)\frac{64}{27\pi^2}\frac{1}{nr^5}e^{-f(0)}\sum_{i=0}^{k}\frac{(f(0))^i}{i!}  \\
  &= &o(1).
\end{eqnarray*}
Therefore,
$$
   n\int_{0}^{\frac r2} \frac{(f(t))^ke^{-f(t)}}{k!}dt
    \sim-\frac{1}{a'\left(\frac{r}{2}\right)}e^{-f\left(\frac r2\right)}\frac{f\left(\frac r2\right)^k}{k!}\sim
   \frac{4}{3\pi}e^{-\frac{2\xi}{3}}\left(\frac{5\pi}{16}\right)^{\frac23}\left(\frac{2}{3}\right)^k\frac{1}{k!}.
$$

\end{proof}

\section{Proof of Proposition~\ref{proposition1}}

\begin{figure}[htbp]
\centering
\includegraphics[width=5cm]{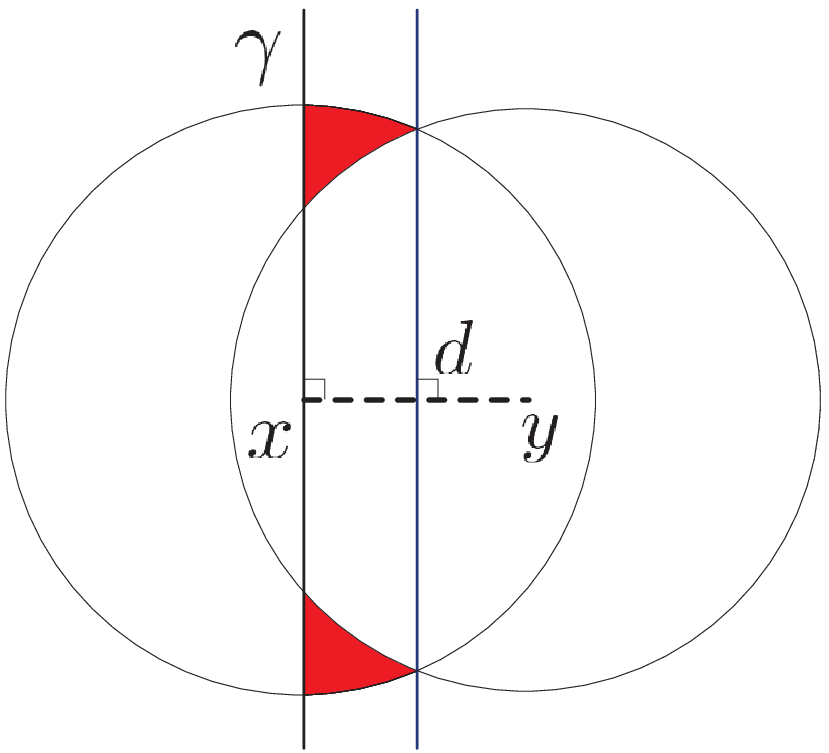}
\caption{Illustration of Lemma~\ref{lem:shadow-low-bound} }
\label{fig:Lemma3}
\end{figure}
We first give a lemma.
\begin{lemma}\label{lem:shadow-low-bound}
The distance  between the centers $x$ and $y$  of two balls which have the same radium $r$ is less than the radius,  i.e., $d=\|xy\|<r$. Plane $\gamma$ is perpendicular to line $xy$ and passes through point $x$. The semi-sphere cut by $\gamma$ which contains point $y$ is denoted by $B_{\mathrm{semi}}(x,r)$, then the volume of $B_{\mathrm{semi}}(x,r)\setminus B(y,r)$ (see red part in Figure~\ref{fig:Lemma3}) is
$$V^{*}(d)=|B_{\mathrm{semi}}\setminus B(y,r)|=\frac{1}{4}\pi d^3.$$
\end{lemma}
\begin{proof}
$
V^{*}(d)=\int_0^{\frac{d}{2}}\pi\left((r^2-t^2)-(r^2-(d-t)^2)\right)dt=\frac{1}{4}\pi d^3.
$
\end{proof}

\begin{proof} (Proof for  Proposition~\ref{proposition1}).
Let $
    d_0=\left(\frac{4r}{n^{\frac23}\pi^{\frac23}}\right)^{\frac13},
$
then  $0<d_0= \Theta\left(\frac{(\log n)^\frac{1}{9}}{n^{\frac13}}\right)<r$ for any sufficiently large $n$, and
$
     \frac{n \pi d_0^3}{4}=(n\pi r^3)^{\frac13}.
$
$\forall x\in \Omega$, let
  $$
       \Gamma_1(x)=\left\{y\in B(x,r)\cap \Omega: \mathrm{dist}(y,x)\leq d_0\right\},$$
$$
  \Gamma_2(x)=\left\{ y\in B(x,r)\cap \Omega: \mathrm{dist}(y,x)\geq d_0, \mathrm{dist}(y,\partial \Omega)> (G(\Omega)+1)r^2\right\},
$$
$$
    \Gamma_3(x)=\left\{y\in B(x,r)\cap \Omega: \mathrm{dist}(y,\partial \Omega)\leq (G(\Omega)+1)r^2\right\}.
$$
Here $\mathrm{dist}(y,x)=\|yx\|$, the distance between points $y$ and $x$.
  Constant $G(\Omega)$ is an uniform upper bound given in Lemma~\ref{lemma:G(Omega)}.  Clearly,
$
    B(x,r)\cap\Omega\subset\Gamma_1(x)\cup\Gamma_2(x)\cup\Gamma_3(x),
$
as Figure~\ref{fig:Ball-in-Gammas} illustrates.

\begin{figure}[htbp]
\begin{center}
\subfigure{
\begin{minipage}{6.5cm}
\includegraphics[scale=0.55] {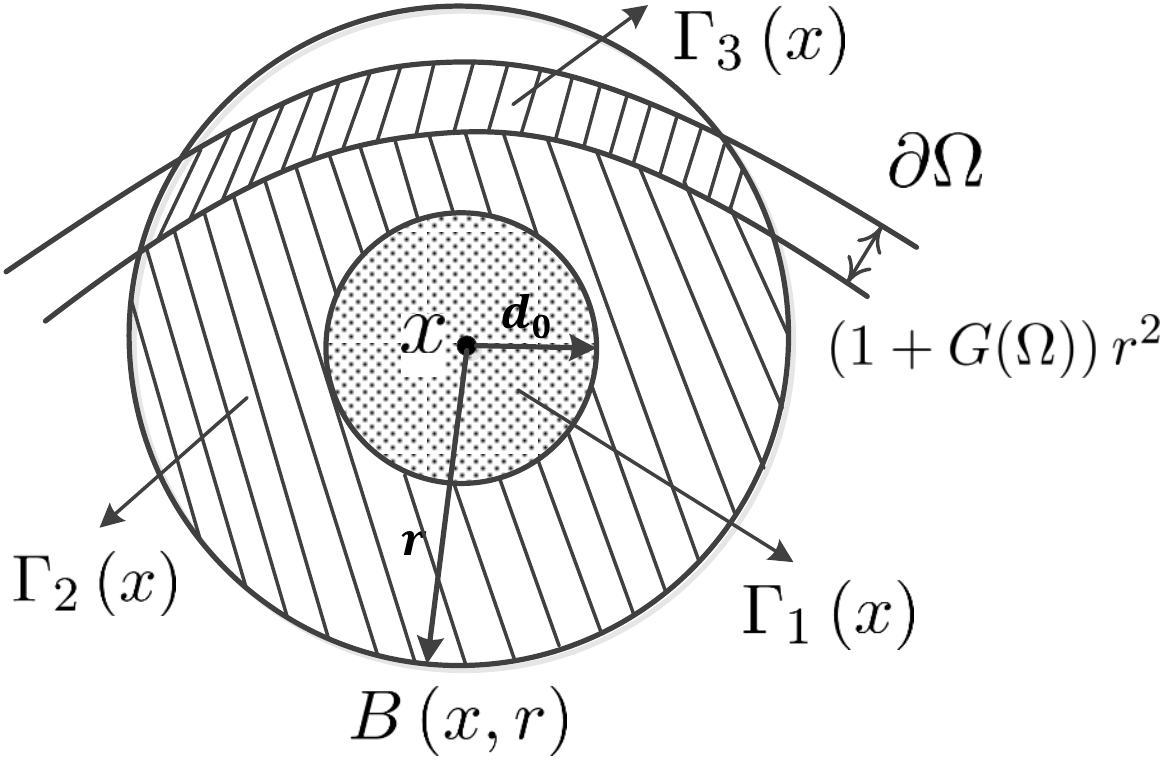}
\caption{ $B(x,r)\cap\Omega\subset\Gamma_1(x)\cup\Gamma_2(x)\cup\Gamma_3(x)$}\label{fig:Ball-in-Gammas}
\end{minipage}
}
\subfigure{
\begin{minipage}{6.5cm}
\includegraphics[scale=0.55]{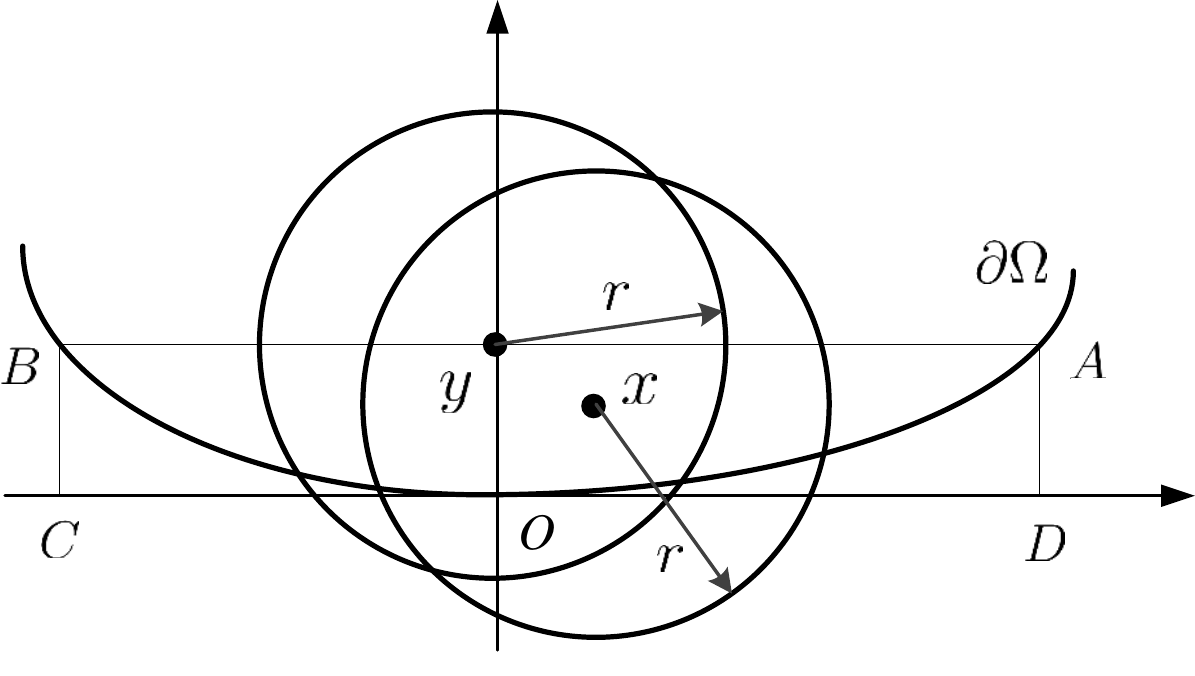}
       \caption{For proof   $A_2=o(1)$}\label{fig:lemma2-Application}
\end{minipage}
}
\end{center}
\end{figure}

%
%
%
%
%\begin{figure}[htbp]
% \begin{center}
%       \scalebox{0.7}[0.7]{
%           \includegraphics{figures/Ball-in-Gammas}       }
%       \caption{Illustration of $B(x,r)\cap\Omega\subset\Gamma_1(x)\cup\Gamma_2(x)\cup\Gamma_3(x)$}\label{fig:Ball-in-Gammas}
%  \end{center}
%  \end{figure}
%
%
%\begin{figure}[htbp]
% \begin{center}
%       \scalebox{0.7}[0.7]{
%           \includegraphics{figures/lemma2-Application} }
%       \caption{Illustration of the proof for $A_2=o(1)$}\label{fig:lemma2-Application}
%  \end{center}
%  \end{figure}

Let
\begin{equation}\label{eq:A123}
      A_i=\frac{1}{n\pi r^3}\int_{\Omega}n\psi^k_{n,r}(x)dx\int_{\Gamma_i(x)}n\Pr(Z_3=k-1)dy, i=1,2,3.
\end{equation}
 We will prove $A_i=o(1), i=1,2,3$, in the following three steps.
  Notice that
we have proved Proposition~\ref{pro:Explicit-Form}: $\int_{\Omega}n\psi^k_{n,r}(x)dx\sim e^{-c}$.

\par Step 1: to prove $A_1=o(1)$.
\begin{eqnarray*}
    A_1&=&\frac{1}{n\pi r^3}\int_{\Omega}n\psi^k_{n,r}(x)dx\int_{\Gamma_1(x)}n\Pr(Z_3=k-1)dy\\
    &\leq&\frac{1}{n\pi r^3}\int_{\Omega}n\psi^k_{n,r}(x)dx\int_{\Gamma_1(x)}ndy\\
     &\leq& \frac{1}{\pi r^3}\int_{\Omega}|\Gamma_1(x)|n\psi^k_{n,r}(x)dx\\
       &\leq& \Theta(1)\frac{ d_0^3}{ r^3}\int_{\Omega}n\psi^k_{n,r}(x)dx =o(1).
\end{eqnarray*}

\par Step 2: to prove $A_2=o(1)$. For any $y\in \Gamma_2(x)$,
$\mathrm{dist}(y,\partial \Omega)\geq(G(\Omega)+1)r^2$. Then by Lemma~\ref{lemma:G(Omega)}, $\|yA\|>r$ and $\|yB\|>r$, see Figure~\ref{fig:lemma2-Application}.
This means that at least more than one half of  $B(y,r)$ falling in $\Omega$.
%So there exists at lest one joint point of $\partial B(y,r)$ and $\partial B(x,r)$ falls in
%$\Omega$.
As a result,
$$v_{y\backslash x}=|B(y,r)\cap \Omega|-|B(x,r)\cap B(y,r)\cap \Omega|\geq \frac{V^{\ast}(d_0)}{2},$$
where  $V^{\ast}(d_0)$ is given by Lemma~\ref{lem:shadow-low-bound}. Therefore,
 $nv_{y\backslash x}\geq\frac{1}{2}nV^{\ast}(d_0)=\frac{n}{2}\frac{\pi d_0^3}{4}=\frac{1}{2}(n\pi r^3)^{\frac13}.$

 \begin{eqnarray*}
    A_2
   &=&\frac{1}{n\pi r^3}\int_{\Omega}n\psi^k_{n,r}(x)dx\int_{\Gamma_2}n\Pr(Z_3=k-1)dy\\
   &\leq&\frac{1}{n\pi r^3}\int_{\Omega}n\psi^k_{n,r}(x)dx
    \int_{\Gamma_2}\frac{n(n\pi r^3)^{k-1}\exp\left(-\frac 12\left(n\pi r^3\right)^{\frac13} \right) dy}{(k-1)!}\\
   &\leq&\frac{1}{n\pi r^3}\int_{\Omega}n\psi^k_{n,r}(x)dx
   \frac{(n\pi r^3)^{k}\exp\left(-\frac {1}{2}\left(n\pi r^3\right)^{\frac13} \right)}{(k-1)!}\\
   &\sim&\frac{e^{-c}}{n\pi r^3}o(1)=o(1).
 \end{eqnarray*}

\par Step 3: to prove $A_3=o(1)$.
 Notice that $\Gamma_3(x)$ falls in a region with the width less than $(G(\Omega)+1)r^2$, we have
$     |\Gamma_3(x)|\leq  \mathrm{Area}(\partial B(x,r)) (G(\Omega)+1)r^2=\Theta(1)r^4.  $
\begin{eqnarray*}
    A_3&=&\frac{1}{n\pi r^3}\int_{\Omega}n\psi^k_{n,r}(x)dx\int_{\Gamma_3(x)}n\Pr(Z_3=k-1)dy\\
      &\leq &\frac{1}{n\pi r^3}\int_{\Omega}n\psi^k_{n,r}(x)dx\int_{\Gamma_3(x)}ndy\\
      &=& \frac{1}{\pi r^3}\int_{\Omega}|\Gamma_3(x)|n\psi^k_{n,r}(x)dx\\
       &\leq&   \Theta(r) \int_{\Omega}n\psi^k_{n,r}(x)dx=o(1).
\end{eqnarray*}
Finally, we have
$$
 \frac{1}{n\pi r^3}\int_{\Omega}n\psi^k_{n,r}(x)dx\int_{\Omega\cap
B(x,r)}n\Pr(Z_3=k-1)dy\leq A_1+A_2+A_3=o(1).$$

The proposition is therefore proved.
\end{proof}

\end{document}